\documentclass[12pt]{amsart}
\usepackage[active]{srcltx}
\usepackage{a4wide}
\usepackage{amsthm,amsfonts,amsmath,mathrsfs,amssymb}
\usepackage{dsfont}
\usepackage[T1]{fontenc}
\usepackage[utf8]{inputenc}
\usepackage{fourier}
\usepackage{pdfpages}
\usepackage{graphicx}
\usepackage{fancybox}

       \def\b{\beta}        \def\g{\gamma}
       \def\la{\lambda}     
              \def\f{\phi}
                  \def\z{\zeta}
\def\ch{\chi}               \def\ta{\tau}

\def\e{\varepsilon}       \def\vf{\varphi}

\def\G{\Gamma}

\def\D{{\mathbb D}}     \def\T{{\mathbb T}}
\def\C{{\mathbb C}}     \def\N{{\mathbb N}}
     
\def\U{{\mathbb U}}
   
   \def\cd{{\mathcal D}}
\def\ch{{\mathcal H}}

\def\({\left(}       \def\){\right)}
\def\pd{\partial}

\newtheorem{prop}{\sc Proposition}
\newtheorem{lem}[prop]{\sc Lemma}
\newtheorem{thm}[prop]{\sc Theorem}

\begin{document}
\title[Co-isometric weighted composition operators]{Co-isometric weighted composition operators on Hilbert spaces of analytic functions}
\author[M.J. Mart\'{\i}n]{Mar\'{\i}a J. Mart\'{\i}n}
\address{Departamento de An\'alisis Matem\'atico, Universidad de La Laguna, 38200 San Crist\'obal de La Laguna, S/C de Tenerife, Spain} \email{maria.martin@ull.es}
\author[A. Mas]{Alejandro Mas}
\address{Departamento de Matem\'aticas, Universidad Aut\'onoma de
Madrid, 28049 Madrid, Spain}
\email{alejandro.mas@uam.es}
\author[D. Vukoti\'c]{Dragan Vukoti\'c}
\address{Departamento de Matem\'aticas, Universidad Aut\'onoma de
Madrid, 28049 Madrid, Spain} \email{dragan.vukotic@uam.es}
\subjclass[2010]{46E22, 47B33}
\keywords{Reproducing kernel, weighted Hardy spaces, weighted composition operator, unitary operator, co-isometric operator.}
\date{20 June, 2020.}
\begin{abstract}
We obtain a necessary and sufficient condition for a weighted composition operator to be co-isometric on a general weighted Hardy space of analytic functions in the unit disk whose reproducing kernel  has the usual natural form. This turns out to be equivalent to the property of being unitary. The result reveals a dichotomy identifying a specific family of weighted Hardy spaces as the only ones that support non-trivial operators of this kind.
\end{abstract}
\maketitle
\section{Introduction}
 \label{sec-intro}
\subsection{Weighted composition operators}
 \label{subsec-WCO}
Let $\D$ denote the unit disk in the complex plane. For a function $F$ analytic in $\D$ and an analytic map $\f$ of $\D$ into itself, the \textit{weighted composition operator\/} (or simply WCO)  $W_{F,\f}$ with \textit{symbols\/} $F$ and $\f$ is defined formally by the formula $W_{F,\f}f = F (f\circ\f) = M_F C_\f f$ as the composition followed by multiplication. Such operators have been studied a great deal for various reasons. Besides their connections with some important problems it is well-known that, in analogy with the classical theorem of Banach and Lamperti, the surjective linear isometries of all (non-Hilbert) Hardy and weighted Bergman spaces are operators of this type \cite{F}, \cite{Kl}.
\subsection{Co-isometric and unitary operators}
 \label{subsec-coisom}
A \textit{linear isometry\/} of a Banach space is a linear operator $T$ such that $\|Tf\|=\|f\|$ for all $f$ in the space. On a Hilbert space this is equivalent to $T\sp{\ast} T=I$; if the Hilbert space isometry is also onto, it is called a \textit{unitary operator\/} and is characterized by the property $T\sp{\ast} T=T T\sp{\ast}=I$. In general, Hilbert spaces have plenty of unitary operators (\text{e.g.}, permuting the elements of an orthonormal basis gives such transformations) so it is of interest to know when an operator of some specific type (such as WCO) is unitary. Here we study the question of when a WCO has the (apparently weaker) property of being \textit{co-isometric\/}: $T T\sp{\ast}=I$, meaning that $T\sp{\ast}$ is an isometry. We show that in this case the properties of being co-isometric and being unitary turn out to be equivalent.
\subsection{Some recent results}
 \label{subsec-recent-res}
Isometric multiplication operators on the Hardy spaces, weighted Bergman spaces, or weighted Dirichlet spaces were characterized in \cite{ADMV}. Characterizations of composition operators on the Dirichlet space that are unitary (which is simple) and isometric (which is more involved) were given in \cite{MV}. Isometric WCOs on weighted Bergman spaces have been recently described by Zorboska \cite{Z1}. Isometric WCOs on non-Hilbert weighted Bergman spaces were discussed in Matache's paper \cite{M}, expanding upon the classical work \cite{F}. Li \textit{et al.\/} \cite{LNNSW} studied normal WCOs on weighted Dirichlet spaces.
\par
Bourdon and Narayan \cite{BN} studied the normal WCOs and described the unitary WCOs on the Hardy space $H^2$. Le \cite{L} considered the  WCOs on a general class of weighted Hardy spaces, denoted by $\ch_\g$, whose reproducing kernel is of the form $(1-\overline{w} z)^{-\g}$ and which enjoy certain conformal properties. Besides describing when an adjoint of such an operator is again of the same type, he characterized the unitary operators among them (in the context of the unit ball), showing that this is equivalent to the property of being co-isometric. Zorboska \cite{Z2} proved several related general results for Hilbert spaces of holomorphic functions in several variables defined by certain properties of their kernels. It should also be mentioned that Hartz \cite{H} obtained several closely related results pertaining to the spaces with kernels of the type mentioned.
\subsection{Statement of our main result}
 \label{subsec-summary-res}
Focusing on the case of one variable and the unit disk, we generalize several results mentioned above by describing the co-isometric operators among the WCOs on the general weighted Hardy spaces whose kernel has the form $\sum_n \g (n) (\overline{w} z)^n$ with $\g(n)>0$. This includes many other spaces besides the $\ch_\g$ mentioned above (and also many spaces in one variable other than those satisfying the properties considered in \cite{Z2}; \textit{cf.\/} p.~156 and p.~163 there).
\par
Our main result, Theorem~\ref{thm-main}, shows that a bounded WCO in any of the spaces considered is co-isometric if and only if it is unitary, if and only if one of the following cases occurs:
\begin{itemize}
 \item
$\f$ is a disk automorphism and $F$ is determined by an explicit  formula, depending on $\f$ and the reproducing kernel, precisely when $\ch$ is one of the spaces $\ch_\g$,
 \item
for all the remaining spaces considered, the WCO must be of trivial type: $\f$ is a rotation and $F$ is a constant function of modulus one.
\end{itemize}
\par\noindent
This shows a sharp contrast between the spaces $\ch_\g$ and the remaining ones.
\par
Because of the generality considered here and the fact that the formula for the kernel is not as explicit as in the usual cases studied in the earlier papers, several standard tools are no longer available. Hence the proofs become more involved and additional arguments are needed. Actually, we split the proof into different partial results because of its length and according to the pertinent cases. An important issue in the proofs is whether the reproducing kernel is bounded (\textit{cf.\/} Theorem~\ref{thm-bded-ker}) or unbounded on the diagonal of the bidisk (\textit{cf.\/} Theorem~\ref{thm-phi-aut} and Theorem~\ref{thm-unbded-ker}).
\par
Another novel point of the present paper is the use of compositions of WCOs in order to ``move points around the disk'' (Lemma~\ref{lem-autom-change}) in order to produce a recurrent formula obtained from a power series expansion, valid on a whole interval instead of just at one point as obtained in the initial stage of the proof of Theorem~\ref{thm-unbded-ker}. This leads to an explicit identification of the spaces from the family $\ch_\g$ as the only spaces that allow non-trivial examples of co-isometric WCOs.
\subsection{Some remarks}
 \label{subsec-rks}
Since every unitary operator is invertible, results on invertibility of WCOs could in principle be relevant in this context. Most recently, in \cite{AV} two different theorems were proved, showing that in every functional Banach space in the disk that satisfies one of the two sets of five axioms listed there, a WCO is invertible if and only if its symbols $F$ and $\f$ have certain properties that one would naturally expect. However, these theorems do not apply to every natural space of analytic functions. One example of a Hilbert space that satisfies these five axioms but is not included among the spaces considered in \cite{L} is the Dirichlet space. Nonetheless, it can be seen that the invertibility results which give a lot of information on the symbols, even in this special case of the Dirichlet space, still require additional non-trivial work so as to deduce the complete information about the exact structure of $F$ and $\f$. It is precisely this work that will be done here. We will make all the proof self-contained although the methods used by Hartz \cite{H} could also give alternative proofs of some of our results.

\section{A review of weighted Hardy spaces}
 \label{sect-RKHS}

\subsection{An important family of weighted Hardy spaces}
 \label{subsec-w-hardy}
Before discussing the general weighted Hardy spaces $\ch$ defined axiomatically, we first single out among them an important family of spaces $\ch_\g$ whose reproducing kernel is given by the formula
\begin{equation}
 K_w(z)=\frac{1}{(1-\overline{w} z)^\g} = \sum_{n=0}^\infty \g (n) (\overline{w} z)^n\,, \qquad z, w \in\D\,,
 \label{eqn-spec-ker}
\end{equation}
where the sequence $(\g(n))_{n=1}^\infty$ is defined as
\begin{equation}
 \g(0)=1\,, \quad \g(1)=\g>0\,, \qquad \g(n) = \binom{\g+n-1}{n} = \frac{\G(n+\g)}{\G(\g) n!} = \frac{1}{(n+\g) B(\g,n+1)}\,, \quad n\ge 1\,.
 \label{eqn-coeff-wds}
\end{equation}
This scale of spaces is formed precisely by the following well-known spaces:
\begin{itemize}
 \item
the standard Hardy space $H^2$, choosing $\g=1$, which yields $\g(n)=1$ for all $n\ge 0$, with the standard Szeg\H{o}
(Riesz) kernel $K_w(z) = (1-\overline{w} z)^{-1}$;
 \item
the (larger) standard weighted Bergman spaces $A^2_{\g-2}$, when $\g>1$, where the usual Bergman space $A^2$ corresponds to the values $\g(n)=n+1$ for $n\ge 0$, and the kernel is the standard Bergman kernel $K_w(z) = (1-\overline{w} z)^{-2}$;
 \item
the (smaller) weighted Dirichlet spaces $\cd_{\g}$, when $\g<1$.
\end{itemize}
Note that the standard unweighted Dirichlet spaces $\cd$ of the functions with square-integrable derivative (with respect to the area measure) does not belong to this scale.
\par\smallskip
Thanks to the nice conformal properties of the spaces $\ch_\g$ (such as the one formulated in \cite[Proposition~4.3]{H}, it can be shown that whenever $\f$ is a disk automorphism (an injective analytic function from $\D$ onto itself) and $F=\mu (\f^\prime)^{\g/2}$, for some constant $\mu$ with $|\mu|=1$, the induced WCO $W_{F,\f}$ is unitary; see \cite[Proposition~3.1]{L}. Moreover, building upon previous results by Bourdon and Narayan \cite{BN}, T. Le \cite[Corollary~3.6]{L} showed that these are the only co-isometries (equivalently, the only unitary operators) among the WCOs on the spaces $\ch_\g$. While he did this in a more general context of the unit ball of $\C^n$ and only for these spaces, here we will limit our attention to the case of the unit disk but considering much more general spaces.
\par\medskip

\subsection{General weighted Hardy spaces}
 \label{subsec-weighted-H2}
Every Hilbert space $\ch$ of analytic functions in $\D$ on which all point evaluations are bounded has reproducing kernels. Given $w\in\D$, the \textit{reproducing kernel\/}  $K_w\,\colon\,\D\to\C$ is a function such that $f(w)=\langle f,K_w \rangle$ for all $f\in\ch$. Two important historical references on reproducing kernels are \cite{A} and \cite{B}. Reproducing kernels are often viewed as functions of two complex variables (defined in the bidisk $\D\times\D$) by writing $K_w(z)=K(z,w)$. Clearly, $K_w (z)=\overline{K_z (w)}$, for all $z$, $w\in\D$.
\par
Obviously, $K_z(z)=\langle K_z,K_z\rangle = \|K_z\|^2\ge 0$ for all $z\in\D$. Actually, the restriction $K_z(z)$ to the diagonal $\{(z,w)\,\colon\,z=w\}$ of the bidisk $\D \times \D$ is often a radial function (meaning that it depends on $|z|$ only). Such spaces are of special interest. In what follows, we will consider a large family of Hilbert spaces $\ch$ of analytic functions in the disk from which only the following axioms will be required:
\begin{itemize}
 \item[(\textbf{A1})]
The point evaluations are bounded (hence $\ch$ is a reproducing kernel Hilbert space);
 \item[(\textbf{A2})]
The reproducing kernel $K_w(z)$ is normalized so that $K_w(0)=1$ for all $w\in\D$;
 \item[(\textbf{A3})]
The monomials $\{z^n\,\colon\, n=0,1,2,\ldots\}$ belong to $\ch$ and form a complete orthogonal set (in the usual sense of maximal orthogonal sets in spaces with inner product).
\end{itemize}
As Proposition~\ref{prop-equiv-ax} formulated below will show, if the above conditions are fulfilled then $\ch$ will also satisfy several other conditions. Among them is the following representation of the reproducing kernel for $\ch$:
\begin{equation}
 K_w(z) = \sum_{n=0}^\infty \g (n) (\overline{w} z)^n
 \label{eqn-ker}
\end{equation}
for certain numbers $\g(n)>0$, where $\g(0)=1$. Actually, computing the inner product of the monomial $z^n$ with the kernel easily yields that
$\g (n) = \|z^n\|^{-2}$. Note also that the restriction $K_z(z) = \sum_{n=0}^\infty \g (n) |z|^{2n}$ is a positive and increasing function of $|z|$. By our normalization $(\mathbf{A2})$, we also have $\g(0) = K_w (0) = 1$; in view of $\g (0) = \|\mathbf{1}\|^{-2}$, where $\mathbf{1}$ denotes the constant function one, it also follows that $\|\mathbf{1}\| = 1$.
\par
As can be seen (and will actually be hinted in some proofs below), the representation of the kernel readily allows for the computation of \cite[Section~2.1]{CM} that gives the norm of a function in $\ch$:
$$
 \|f\|^2 =\sum_{n=0}^\infty \frac{1}{\g(n)}|a_n|^2\,, \quad \mathrm{whenever \ }  f(z)=\sum_{n=0}^\infty a_n z^n\,.
$$
The spaces $\ch$ considered here are often called \textit{weighted Hardy spaces}. These spaces and some important operators on them were studied in detail by Shields \cite{S} and later  by many followers.
\par
Regarding the notation used, it is worth mentioning that here we are centered on the role played by the kernels and therefore use mainly the numbers $\g(n)$ whereas in many other texts the emphasis is on the norm formula in terms of the Taylor coefficients, so these spaces are denoted there by $H^2(\b)$, where the obvious relationship between the numbers $\b(n)$ and $\g(n)$ is as follows:
$$
 \g(n)= \frac{1}{\beta(n)^2}\,, \qquad n\ge 0\,.
$$
We refer the reader to the standard reference \cite{CM}.

\subsection{Consequences and restatements of the axioms}
 \label{subsec-equiv-ax}
To fix the notation, the rotations will be denoted by $R_\la$; for $|\la|=1$, let $R_\la(z)=\la z$, for all $z\in\D$. The induced composition operator is denoted by $C_{R_\la}$: $C_{R_\la}f=f\circ R_\la$, for $f\in\ch$.
\par
The following simple result shows that one does not need to assume any further axioms that our spaces should satisfy in order to obtain the results that will be proved here. Moreover, it shows that under only a  minimum set of assumptions we can produce examples of unitary WCOs on our spaces.
\par
For our purpose it would actually suffice just to note that condition (a) below implies all the others (since this is all that is needed in the results that follow). This is either simple to prove or easy to find in other places; for example, (a) $\Rightarrow$ (b) is proved in \cite[Theorem~2.10]{CM}. However, we have chosen to include a detailed proof of equivalences between all the conditions (a)--(g) below because this may have some independent interest. Since this result is not easy to find in one place and some of the implications are non-trivial, for the sake of completeness we also include a detailed proof, as self-contained as possible.
\begin{prop} \label{prop-equiv-ax}
Let $\ch$ be a Hilbert space of analytic functions in $\D$ that contains all monomials and satisfies our axioms $(\mathbf{A1})$ and $(\mathbf{A2})$ from Subsection~\ref{subsec-weighted-H2}: the point evaluations are bounded on $\ch$ and the reproducing kernel is normalized so that $K_w(0)=1$ for all $w\in\D$. Then the following statements are equivalent:
\begin{itemize}
\item[(a)]
 Axiom $(\mathbf{A3})$ is fulfilled; that is, the monomials $\{z^n\,\colon\,n=0,1,2,\ldots\}$ form a complete orthogonal set in $\ch$.
\item[(b)]
 The reproducing kernel has the form \eqref{eqn-ker}: $K_w(z) = \sum_{n=0}^\infty \g (n) (\overline{w} z)^n$, with $\g (n) = \|z^n\|^{-2}$.
\item[(c)]
 The norm of a function $f\in\ch$ whose Taylor series in $\,\D$ is $f(z)=\sum_{n=0}^\infty a_n z^n$ is given by
$$
 \|f\|^2= \sum_{n=0}^\infty |a_n|^2 \|z^n\|^2\,.
$$
\item[(d)]
 The rotations $R_\la$ induce isometric composition operators $C_{R_\la}$ on $\ch$.
\item[(e)]
 The rotations $R_\la$ induce unitary composition operators $C_{R_\la}$ on $\ch$.
\item[(f)]
 The constant multipliers of modulus one ($|\mu|=1$) and rotations $R_\la$ induce unitary weighted composition operators $W_{\mu,\,R_\la}$ on $\ch$.
\item[(g)] $K_{\la w} (\la z) = K_w (z)$ for all $z$, $w\in\D$ and all $\la$ with $|\la|=1$.
\end{itemize}
\end{prop}
\begin{proof}
It suffices to prove the following chain of implications:
$$
 \mathrm{(b)} \Rightarrow \mathrm{(a)} \Rightarrow \mathrm{(c)} \Rightarrow \mathrm{(d)} \Rightarrow \mathrm{(e)} \Rightarrow \mathrm{(g)} \Rightarrow \mathrm{(b)},
$$
\par\noindent
in addition to the easy implications \  (e) $\Rightarrow$ (f) $\Rightarrow$ (d).
\par\medskip
\shadowbox{(b) $\Rightarrow$ (a):} \ Assume that the kernel can be written as $K_w(z) = \sum_{n=0}^\infty \g (n) (\overline{w} z)^n$, with $\g (n) = \|z^n\|^{-2}$. We first check that the series converges in the norm of the space $\ch$. To this end, write
$$
 p_{w,N}(z)= \sum_{n=0}^N \gamma(n) \overline w^n z^n\,, \quad z, w\in\D\,, \qquad N\ge 0\,,
$$
for the partial sums of the kernel $K_w(z)$. Note that for $a\in\D$ we have $K_a(a) = \sum_{n=0}^\infty \gamma(n) |a|^{2n}$. Then clearly the infinite series $\,\sum_{n=0}^\infty |w|^n\/$ and $\,\sum_{n=0}^\infty \gamma(n) |w|^{n}\/$ both converge for all $w\in\D$. Given $\e>0$, let $N\in\N$ be such that $\,\sum_{n>N} |w|^n<\varepsilon/2$ and $\,\sum_{n>N} \gamma(n) |w|^{n}<\varepsilon/2$. Then
\begin{eqnarray*}
 \|K_w-p_{w,N}\| &=& \left\| \sum_{n>N} \gamma(n) \overline{w}^n z^n  \right\| \leq  \sum_{n>N} \gamma(n) |w|^{n} \|z^n\|
\\
 &=& \sum_{n>N,\,\|z^n\|>1}  \gamma(n) |w|^{n} \|z^n\| + \sum_{n>N,\, \|z^n\|\le 1}  \gamma(n) |w|^{n} \|z^n\|
\\
 &\le & \sum_{n>N,\,\|z^n\|>1}  \frac{|w|^{n}}{\|z^n\|} + \sum_{n>N}  \gamma(n) |w|^{n}
\\
 &\le &\sum_{n>N}  |w|^{n} + \sum_{n>N}  \gamma(n) |w|^{n}
 <\varepsilon\,.
\end{eqnarray*}
Now, using the definition of the reproducing kernel $K_w$ (as a function of $z$), the following operation is justified for all $w\in\D$:
$$
 w^n = \langle z^n, K_w(z) \rangle = \langle z^n, \sum_{m=0}^\infty \g(m) \overline{w}^m z^m\rangle = \sum_{m=0}^\infty \overline{\g(m)} \langle z^n,z^m\rangle w^m \,, \qquad n\ge 0\,.
$$
The uniqueness of the Taylor coefficients implies $\langle z^m,z^n\rangle=0$ for all $m\neq n$. Hence the monomials form an orthogonal system.
\par
It is only left to check that this orthogonal system is complete. To this end, it suffices to note that the linear span of the kernels is dense in $\ch$: if a function is orthogonal to it, in particular it is orthogonal to each $K_w$ and hence vanishes at all $w\in\D$, so it must be the zero function. We have proved a little earlier that every kernel, and hence every function in the linear span of the kernels, can be approximated by polynomials in the norm of the space. Thus, the polynomials are dense in $\ch$ and therefore form a complete orthogonal system.
\par\medskip
\shadowbox{(a) $\Rightarrow$ (c):} \ Let $f\in\ch$, with the Taylor series $f(z)=\sum_{n=0}^\infty a_n z^n$, $z\in\D$. By assumption, the monomials form a complete orthogonal set in $\ch$. As a consequence of our Axiom~(\textbf{A1}), the development of $f$ in a series with respect to the orthonormal basis $\{z^n/\|z^n\|\,:\,n\ge 0\}$ must coincide with the Taylor series of $f$ at each point $z\in\D$. This  justifies the following operations:
$$
\|f\|^2= \langle \sum_{m=0}^\infty a_m z^m, \sum_{n=0}^\infty a_n z^n \rangle = \sum_{m=0}^\infty \sum_{n=0}^\infty a_m \overline{a_n}  \langle z^m, z^n \rangle = \sum_{n=0}^\infty |a_n|^2 \|z^n\|^2\,.
$$
\par\medskip
\shadowbox{(c) $\Rightarrow$ (d):} \ Let $\lambda$ be a complex number of modulus one. Since $C_{R_\lambda}f=f(\lambda z)$, by the norm formula from the assumption (c) we have
$$
 \|C_{R_\lambda}f\|^2 = \sum_{n=0}^\infty |\la^n a_n|^2 \|z^n\|^2 = \|f\|^2\,.
$$
\par\medskip
\shadowbox{(d) $\Rightarrow$ (e):} \ It suffices to notice that $C_{R_\la} C_{R_{\overline{\la}}} f=f$ for all $f\in\ch$, hence $C_{R_\la}$ is onto. Being a surjective isometry, it is a unitary operator.
\par\medskip
\shadowbox{(e) $\Rightarrow$ (g):} \ Let  $|\lambda|=1$. Since $C_{R_{\lambda}}$ is unitary, we know that $C_{R_{\lambda}}C^{*}_{R_{\lambda}}=I$. Moreover,
$$
 C^{*}_{R_{\lambda}}K_{w}(z)=\langle C^{*}_{R_{\lambda}}K_{w}, K_{z} \rangle=\langle K_{w}, C_{R_{\lambda}}K_{z} \rangle=\overline{\langle C_{R_{\lambda}}K_{z}, K_{w}\rangle }=\overline{C_{R_{\lambda}}K_{z}(w)}=\overline{K_{z}(\lambda w)}=K_{\lambda w}(z)\,.
$$
Hence
$$
 K_{w}(z)=C_{R_{\lambda}}C^{*}_{R_{\lambda}}K_{w}(z)=C_{R_{\lambda}}K_{\lambda w}(z)=K_{\lambda w}(\lambda z).
$$
\par\medskip
\shadowbox{(g) $\Rightarrow$ (b):} \ It is a standard Hilbert space fact that the reproducing kernel $K(z,w)=K_w(z)$ is a positive definite function (in the usual sense that the corresponding quadratic form is positive semi-definite). We also have the normalization condition, Axiom (\textbf{A2}). Condition (g) tells us that
$$
 K_{\la w}(\la z)=K_w(z)\,,\quad z, w\in\D\,,\ |\la|=1\,,
$$
and since the rotations are easily verified to be the only $\C$-linear unitary maps of $\C$. Thus, we can apply \cite[Lemma~2.2]{H} to conclude that there exists a function $h$ analytic in $\D$ such that
$$
 K_w(z)=h(\overline w z)\,,\quad h(z)=\sum_{n=0}^\infty \gamma (n) z^n\,,\quad z, w\in\D\,, \quad \gamma (0)=1\,, \ \gamma (n)\ge 0.
$$
We still ought to show that $\g(n)=\|z^n\|^{-2}$ for all $n$, so further work is needed.
\par
We follow the argument of \cite[Proposition~4.1]{GHX} and include the details for the sake of completeness. Let $I=\{n\in\N\,:\,\gamma (n)>0\}$, so that $h(z)=\sum_{n\in I} \gamma (n) z^n$. Consider the Hilbert space $H$ of analytic functions in the disk with orthonormal basis $\left\{ e_n(z) = \sqrt{\gamma (n)} z^n\,:\,n\in I\right\}$. Our next objective is to show that the spaces $H$ and $\ch$ coincide.
\par
We first show that $H$ is a reproducing kernel Hilbert space by checking that the point evaluation functionals are bounded. Let $g\in H$, with $g=\sum_{n\in I} b_n e_n$. Then, given $w\in \D$, the Cauchy-Schwarz inequality yields
$$
 |g(w)| \le \left(\sum_{n\in I} |b_n|^2\right)^{1/2} \left(\sum_{n\in I} |e_n(w)|^2\right)^{1/2}
 = \left(\sum_{n\in I} |b_n|^2\right)^{1/2} \left(\sum_{n\in I} \gamma (n) |w|^{2n}\right)^{1/2} = \sqrt{h(|w|^2)} \|g\|_H\,.
$$
\par
Next, we show that all functions in $H$ are analytic in $\D$. Since $h$ is analytic in $\D$ and has non-negative Taylor coefficients, $h(|z|)$ is an increasing function of $|z|$. Boundedness of point evaluations shows that convergence in $H$ implies uniform convergence on compact subsets of $\D$. Since the orthonormal basis of $H$ consists of polynomials, each function in $H$ is a uniform limit of polynomials on compact subsets of $\D$ and, hence, an analytic function in $\D$.
\par
Finally, we check that the reproducing kernels of $H$ and $\ch$ coincide. Denote by $E_w$ the reproducing kernel in $H$ for the point evaluation at $w\in\D$: $g(w)=\langle g, E_w \rangle_H$, for $g\in H$. If $E_w(z)=\sum_{n\in I} c_n e_n(z)$, for each $m\in I$ we have
$$
 e_m(w) = \langle e_m,E_w \rangle_H = \sum_{n\in I} \overline{c_n} \langle e_m,e_n \rangle_H = \overline{c_m}\,,
$$
hence
$$
 E_w(z) = \sum_{n\in I} \overline{e_n(w)} e_n(z) = \sum_{n\in I} \gamma (n) \overline{w}^n z^n = K_w(z)
$$
for all $z$, $w\in\D$.
\par
Finally, both $\ch$ and $H$ are Hilbert spaces with positive definite kernels that coincide. By the uniqueness part of the Moore-Aronszajn theorem \cite[p.~344,~(4)]{A}, we must have $H=\ch$.
\par
It is only left to determine the coefficients $\g (n)$ for all $n\ge 0$. For all $n\in I$ we easily compute $1 = \|e_n\|_H^2 = \|e_n\|_\ch^2 = \langle \sqrt{\gamma (n)} z^n,\sqrt{\gamma (n)} z^n \rangle_\ch = \gamma (n) \|z^n\|^2_\ch$, hence $\gamma (n)=\|z^n\|_\ch^{-2}$, as claimed.
\par
Finally, it is only left to show that all the Taylor coefficients of $h$ must be non-zero: $I=\N\cup\{0\}$. Otherwise, there exists an index $N$ such that $\gamma (N)=0$ and, since $z^N\in \ch=H$, we have
$$
 z^N = \sum_{n\in I} b_n e_n(z) = \sum_{n\in I} b_n \sqrt{\g (n)} z^n\,,
$$
with all the values $n\in I$ being different from $N$, which is clearly  impossible.
\par\medskip
\shadowbox{(e) $\Rightarrow$ (f):} \ Trivially, multiplication by a constant of modulus one yields a surjective isometry (hence a unitary operator) and the product of two unitary operators is unitary.
\par\medskip
\shadowbox{(f) $\Rightarrow$ (d):} \ This assertion follows trivially by choosing the constant multiple to be one.
\end{proof}

\section{Characterizations of co-isometric WCOs}
 \label{sect-unit-wco}
\par\medskip
\subsection{Statement of the main theorem}
 \label{subsec-main-res}
\par
We are now ready to formulate our main result.
\par
\begin{thm} \label{thm-main}
Let $\ch$ be a weighted Hardy space that satisfies axioms (\textbf{A1})--(\textbf{A3}) and let $W_{F,\f}$ be a bounded WCO on $\ch$, where $F\/$ is a function analytic in $\,\D$ and $\f\/$ an analytic map of $\,\D$ into itself. Then $W_{F,\f}$ is unitary if and only if it is co-isometric (that is, if and only if $\,W_{F,\f} W_{F,\f}\sp{\ast}=I$).
\par
Moreover, any of these two properties is further equivalent to the following:
\begin{itemize}
\item[(a)]
$\/\f$ is a disk automorphism and $F = \mu (\f^\prime)^{\g/2} = \nu \frac{K_a}{\|K_a\|}$, where $a=\f^{-1}(0)$ and $\mu$ and $\nu$ are  constants such that $|\mu|=|\nu|=1$, in the case when $\ch$ is one of the spaces $\ch_\g$ considered above.
\item[(b)]
 $\/\f$ is a rotation and $F$ is a constant function of modulus one, whenever $\ch$ does not belong to the scale of spaces $\ch_\g$.
\end{itemize}
\end{thm}
We shall refer to the operators given in the case (b) as to the trivial ones. It is interesting to notice that there are results in the literature with a similar flavor, although in a somewhat different context; \textit{cf.\/}, for example, Proposition~4.3 and Corollary~9.10 of \cite{H}. The recent paper \cite{Z2} for a general class of spaces in the context of several variables (where certain conformal properties of the kernels are again assumed) contains some  related ideas and similar results.
\par
The rest of the paper is devoted to the proof of this result which we split up into a sequence of auxiliary statements in order to make it easier to follow. We begin by seeing that the assumption that $W_{F,\f}$ is a co-isometry imposes certain important properties of the symbols.
\par
\subsection{Basic information on the symbols of $W_{F,\f}$}
 \label{subsec-fla-symb}
As a preliminary information that will be needed later, we observe the following. Using the basic properties of the inner product and applying the operator to a reproducing kernel, we obtain
\begin{eqnarray*}
 \langle f, W_{F,\f}^* K_w\rangle & = & \langle W_{F,\f} f, K_w\rangle
= \langle F (f\circ \f), K_w\rangle = F(w) f(\f(w))
\\
 & = & F(w) \langle f, K_{\f(w)}\rangle = \langle f, \overline{F(w)} K_{\f(w)}\rangle\,, \qquad f\in\ch\,, \quad w\in\D\,.
\end{eqnarray*}
Hence we have the formula for the action of the adjoint of a WCO on reproducing kernels:
\begin{equation}
 W_{F,\f}^* K_w = \overline{F(w)} K_{\f(w)}\,, \qquad w\in\D\,.
 \label{eqn-adj-ker}
\end{equation}
Using this, we can show that the assumption that a WCO is co-isometric already provides some rigid information on the symbols $F$ and $\f$. As is usual, by a \textit{univalent function\/} in $\D$ we mean a function analytic in the disk which is one-to-one there.
\par
\begin{prop} \label{prop-F}
Let $\ch$ be a weighted Hardy space that satisfies axioms (\textbf{A1})--(\textbf{A3}) and let $W_{F,\f}$ be a bounded WCO on $\ch$. If $\,W_{F,\f}$ is co-isometric then $F$ is given by
\begin{equation}
 F(z) = \frac{1}{\overline{F(0)} K_{\f(0)}(\f(z))}\,, \qquad \mathrm{ \ for \ all \ } z\in\D\,,
 \label{eqn-F}
\end{equation}
and $\f$ is a univalent function.
\end{prop}
\begin{proof}
Since by assumption $W_{F,\f} W_{F,\f}^* = I$, using \eqref{eqn-adj-ker}, we obtain
$$
 K_w = W_{F,\f} W_{F,\f}^* K_w = W_{F,\f} \(\overline{F(w)} K_{\f(w)}\) = \overline{F(w)} W_{F,\f} K_{\f(w)} = \overline{F(w)} F (K_{\f(w)}\circ\f)\,.
$$
Thus,
\begin{equation}
 F(z) \overline{F(w)} K_{\f(w)}(\f(z)) = K_w(z)\,, \qquad \mathrm{ \ for \ all \ } z,w\in\D\,.
 \label{eqn-F-fi}
\end{equation}
Taking $w=0$, the right-hand side is $\equiv 1$, so \eqref{eqn-F} follows immediately.
\par\medskip
We now show that $\f$ is univalent. To this end, let $z_1$, $z_2\in\D$ be such that $\f(z_1)=\f(z_2)$. Then it follows that $F(z_1)=F(z_2)$: indeed, by \eqref{eqn-F} we have
$$
 F(z_1) =\frac{1}{\overline{F(0)} K_{\f(0)}(\f(z_1))} = \frac{1}{\overline{F(0)} K_{\f(0)}(\f(z_2))} = F(z_2)\,.
$$
Next, using \eqref{eqn-F-fi} three times but with different values of $z$ and $w$, we get
\begin{eqnarray*}
 K_{z_1}(z_1) &=& |F(z_1)|^2 K_{\f(z_1)}(\f(z_1))\,,
\\
 K_{z_2}(z_2) &=& |F(z_2)|^2 K_{\f(z_2)}(\f(z_2))\,,
\\
 K_{z_2}(z_1) &=& F(z_1) \overline{F(z_2)} K_{\f(z_2)}(\f(z_1))\,.
\end{eqnarray*}
By our choice of $z_1$ and $z_2$, the right-hand sides of the last three equations all coincide, hence
$$
 K_{z_1}(z_1) = K_{z_2}(z_2) = K_{z_2}(z_1) \,.
$$
Recall that for any kernel given by \eqref{eqn-ker} the function $K_z(z)$ is a strictly increasing function of $|z|$, hence from the first equality above it follows that $|z_1|=|z_2|$. If $z_1=0$, it follows that also $z_2=0$ and we are done. Thus, we may assume that $|z_1|>0$.
\par
Next, writing $z_2=\la z_1$ with $|\la|=1$ and using the remaining equality and equation \eqref{eqn-ker}, we obtain
$$
 \sum_{n=1}^\infty \g (n) |z_1|^{2n} = \sum_{n=1}^\infty \g (n) \overline{\la}^n |z_1|^{2n}\,,
$$
which yields
$$
 \sum_{n=1}^\infty \g (n)\, \mathrm{Re\,} (1-\overline{\la}^n) \, |z_1|^{2n} = 0\,.
$$
Since each term in the sum on the left is non-negative and $\g(n)>0$ and $|z_1|>0$, it follows that $$
 \mathrm{Re\,} (1-\overline{\la}^n) = 0
$$
for all $n\ge 1$, which implies $\la=1$, hence $z_1=z_2$. This proves  that $\f$ is univalent.
\end{proof}
\par
It should be noted that formula \eqref{eqn-F-fi} has already appeared before in the literature; see Zorboska \cite[Proposition~1]{Z2}.

\subsection{Kernels bounded on the diagonal}
 \label{subsec-pf-ker-bded-diag}
Proposition~\ref{prop-F} proved above will help us to handle the simpler case of the kernel bounded on the diagonal. We will refer to $\{(z,w)\in\D\times\D\,\colon\,z=w\}$ as the \textit{diagonal\/} of the bidisk $\D\times\D$. It  will be relevant to our proofs to distinguish between the kernels that are bounded on the diagonal and those that are not.
\par
A simple example of the kernel of type \eqref{eqn-ker} which is bounded on the diagonal is
$$
 K_w(z) = 1 + \overline{w} z + \sum_{n=2}^\infty \frac{\overline{w} z}{n (n-1)} = 1 + 2 \overline{w} z - (1- \overline{w} z) \log \frac{1}{1- \overline{w} z}\,.
$$
\par
In relation to an argument mentioned in the proof below, it is convenient to recall that reproducing kernels may have zeros in the bidisk (on or off the diagonal); this question is relevant in the theory of one and several complex variables. In relation to the kernels considered here, we mention the recent reference \cite{P}.
\par\medskip
\begin{thm} \label{thm-bded-ker}
Let $\ch$ be a weighted Hardy space that satisfies axioms (\textbf{A1})--(\textbf{A3}) and whose reproducing kernel is bounded on the diagonal of the bidisk, and let $W_{F,\f}$ be a bounded WCO on $\ch$. Then the following statements are equivalent:
\par
(a) $W_{F,\f}$ is unitary.
\par
(b) $\,W_{F,\f}$ is co-isometric.
\par
(c)
$F$ is a constant function of modulus one and $\f$ is a rotation.
\end{thm}
\begin{proof}
Trivially, (a) implies (b). It is clear from Proposition~\ref{prop-equiv-ax} that (c) implies (a). Thus, it only remains to show that (b) implies (c).
\par
Suppose that $\,W_{F,\f}$ is co-isometric. The assumption that the kernel is bounded on the diagonal of the bidisk is equivalent to saying that $\sum_{n=0}^\infty \g (n) < +\infty$. The Weierstrass test and formula \eqref{eqn-ker} readily imply that the kernel extends continuously to the closed bidisk $\overline{\D}\times\overline{\D}$. Since in principle the kernel  could have zeros, in view of \eqref{eqn-F} we need an additional argument in order to show that $F$ extends continuously to $\overline{\D}$. This can be seen as follows: for every fixed $z\in\D$, the function $K_{z}$ is continuous in $\overline{\D}$. Note also that
$$
 |F(z)| = |\langle F,K_z\rangle| \le \|F\| \|K_z\| = \|F\| \sqrt{K_z(z)} \le \|F\| \max_{\z\in\overline{\D}} \sqrt{K_\z(\z)} \,,
$$
for all $z\in\D$, hence $F$ is bounded in the disk and therefore also in $\overline{\D}$. Equation \eqref{eqn-F} implies that $K_{\f(0)}(\f(z))$ is bounded away from zero in the disk and, since it is continuous in $\overline{\D}$, it is also bounded away from zero in the closed disk. This, together with \eqref{eqn-F}, shows that $F$ is continuous in $\overline{\D}$.
\par
In view of Proposition~\ref{prop-F}, setting $w=z$ in \eqref{eqn-F-fi}, we know that
\begin{equation}
  |F(z)|^2 = \frac{K_z(z)}{K_{\f(z)}(\f(z))}\,, \qquad z\in\D\,.
  \label{eqn-F2}
\end{equation}
Let $\T=\{z\,:\,|z|=1\}$ denote the unit circle. Since $\f$ is  analytic in $\D$ and bounded by one, the finite radial limits $\f(\z)=\lim_{r\to 1^-} \f(r\z)$ exist and also satisfy $|\f(\z)|\le 1$ for almost every point $\z\in\T$ with respect to the normalized Lebesgue arc length measure on $\T$ \cite{D, Ko}. If $\z\in\T$ is a point where $\f(\z)$ exists, by letting $z\to\z$ radially inside the disk and taking into account that $|\f(\z)|\le 1=|\z|$ and $K_z(z)$ is an increasing function of $|z|$, we see immediately that $|F(\z)|\ge 1$. Hence $|1/F|\le 1$ almost everywhere on $\T$ and therefore $|1/F|\le 1$ in $\D$ (by standard Hardy space arguments). On the other hand, from \eqref{eqn-F2} we get
$$
 |F(0)|^2 = \frac{1}{K_{\f(0)}(\f(0))} \le 1\,,
$$
again because $K_z(z)$ is an increasing function of $|z|$. The maximum modulus principle applied to $1/F$ implies that $F$ is identically constant and has modulus one.
\par
In view of \eqref{eqn-F2} we have $K_z(z)=K_{\f(z)}(\f(z))$ for all $z\in\D$. Taking into account the form \eqref{eqn-ker} of the kernel, this means that
$$
 1 + \sum_{n=1}^\infty \g (n) |z|^{2n} = 1 + \sum_{n=1}^\infty \g (n) |\f(z)|^{2n}\,.
$$
Since $1+\sum_{n=1}^\infty \g (n) r^{2n}$ is a strictly increasing function of $r$, it follows that equality above is possible if and only if $|\f(z)|=|z|$, for every $z\in\D$. But this implies that $\f$ is a rotation.
\end{proof}
\par
\par
\subsection{Kernels unbounded on the diagonal}
 \label{subsec-pf-ker-unbded-diag}
\par
For the kernels of the general form \eqref{eqn-ker} considered here, the assumption that $K_w(z)$ is unbounded on the diagonal obviously means that $\lim_{|z| \to 1^-} K_z(z) = +\infty$, which is easily seen to be  equivalent to $\sum_{n=0}^\infty \g (n) = +\infty$.
\par
It is relevant to note that for any of the spaces $\ch_\g$ defined earlier for which \eqref{eqn-coeff-wds} is satisfied, the reproducing kernel is always unbounded on the diagonal since $\g(1)>0$. An example of a space $\ch$ not in the family $\ch_\g$ and whose kernel is unbounded on the diagonal is the classical Dirichlet space (renormed) with
$$
 \g(n)=\frac{1}{n+1}\,, \qquad K_w(z) = \frac{1}{\overline{w} z} \log \frac{1}{1 - \overline{w} z}\,.
$$
Indeed, it can be checked that \eqref{eqn-coeff-wds} is not fulfilled in this case.
\par
We will see that, unlike in the previous case, for certain kernels of this type non-trivial unitary WCOs will exist.
\par\medskip
\begin{thm} \label{thm-phi-aut}
Let $\ch$ be a weighted Hardy space that satisfies axioms (\textbf{A1})--(\textbf{A3}) and whose reproducing kernel is unbounded on the diagonal. If $\,W_{F,\f}$ is co-isometric on $\ch$ then $\f$ is a disk automorphism.
\par
In the case when $\f(0)=0$ (in particular, whenever $\,\f\/$ is a rotation), $F$ must be a constant of modulus one and the induced operator  $\,W_{F,\f}$ is unitary on $\ch$.
\end{thm}
\begin{proof}
We recall that an inner function is a bounded function with radial limits of modulus one almost everywhere on the unit circle $\T$. From the basic factorization theory of Hardy spaces \cite[Chapter~2 and Theorem~3.17]{D}, it follows that a univalent inner function must be a disk automorphism; note that this can also be concluded by post-composing with disk automorphisms and invoking Frostman's theorem. We already know from Theorem~\ref{thm-bded-ker} that $\f$ is univalent. Thus, it suffices to show that it is also an inner function.
\par
We certainly know that $\f$ is an analytic self-map of $\D$ so it must have radial limits almost everywhere and these limits have modulus at most one. Consider the set
$$
 E = \{\z\in\T\,\colon\,|\f(\z)|<1\}
$$
and show that its arc length measure is $m(E)=0$. Let us look again at formula \eqref{eqn-F2}. If $\z\in E$, since the kernel is unbounded on the diagonal, we have $\lim_{z\to\z} K_z(z)=+\infty$. On the other hand, $\f(\z)\in\D$ by our definition of $E$, hence the value $K_{\f(\z)}(\f(\z))$ is defined and finite. It follows that $F(z)\to\infty$ as $z\to\z$. Now in view of \eqref{eqn-F}, we obtain $K_{\f(0)}(\f(\z)) =0$. Since $\z\in E$ was arbitrary, this holds for all $\z\in E$.
\par
Now assume the contrary to our assumption: $m(E)>0$. First note that, by our definition of $E$, the set $\f(E)$ is contained in $\D$ and clearly
\begin{equation}
 E \subset \bigcup_{a\in\f(E)} \{\z\in\T\,\colon\,\f(\z)=a\}\,.
 \label{eqn-union}
\end{equation}
Next, we claim that $m(\{\z\in\T\,\colon\,\f(\z)=a\})=0$, for every $a\in\f(E)$. We know that $\f$ is univalent, hence it cannot be identically constant. And since $\f\in H^\infty$, it is impossible for $\f(\z)=a$ to hold on a set of positive measure on $\T$ by a theorem of Privalov (see \cite[Theorem~2.2]{D} or \cite[Chapter~III]{Ko} for different versions of it). This proves the claim.
\par
Now we can distinguish between two cases, depending on the cardinality of the set $\f(E)$. If $\f(E)$ is countable then $m(E)=0$ by \eqref{eqn-union}, as claimed. And if $\f(E)$ is uncountable, we argue as follows. The set $\f(E)$ is contained in $\D$ and has at least one accumulation point in $\D$. (Otherwise, each compact disk $D_n = \{z\,\colon\, |z|\le 1 - \frac{1}{n}\}$, $n\in\N$, would contain only finitely many points of $\f(E)$ and since $\D=\cup_{n=1}^\infty D_n$, the set $\f(E)$ would be countable.) But, as we have noticed above, the analytic function $K_{\f(0)}$ vanishes on $\f(E)$ and is therefore identically zero in $\D$ by the uniqueness principle, which is impossible.
\par
Thus, we conclude that $m(E)=0$, hence $\f$ is an inner function. This completes the proof that $\f$ is an automorphism.
\par\smallskip
As for the final part of the statement, if $\f(0)=0$ the function $F$ must be a constant of modulus one: indeed, equation \eqref{eqn-F} together with $K_0(z)\equiv 1$ shows that $F\equiv 1/\overline{F(0)}$. Writing $F\equiv\la$, it is immediate that $|\la|=1$. Now, it follows directly from Proposition~\ref{prop-equiv-ax} that the induced operator  $\,W_{\la,\f}$ is actually unitary.
\end{proof}
\par\smallskip
In what follows we will rely on the properties of the disk automorphisms. One basic type of automorphisms are the rotations  $R_\la$, where $|\la|=1$. Since our spaces are supposed to satisfy the axioms listed in Lemma~\ref{prop-equiv-ax}, it follows that the induced composition operators $C_{R_\la}$ are unitary on $\ch$. The other basic type of automorphisms are the maps $\vf_a(z)=(a-z)/(1-\overline{a}z)$, $a\in\D$; such an automorphism is an involution and exchanges the point $a$ and the origin. As is well-known, every disk automorphism $\f$ is of the form $\f=R_\la \vf_a$. To make the notation more compact, we will write $\vf_{\la,a} = R_\la \vf_a$.
\par
The following lemma will be fundamental in proving the last theorem of this paper. It will allow us to change from one co-isometric WCO to another operator of the same kind (acting on the same space) in a convenient way. We give an elementary proof below. However, it is convenient to remark that a stronger version of the statement could be derived from \cite[Proposition~9.9]{H}.
\par\medskip
\begin{lem} \label{lem-autom-change}
Let $\ch$ be a weighted Hardy space that satisfies axioms (\textbf{A1})--(\textbf{A3}) and let $W_{F,\f}$ be a co-isometric WCO on $\ch$ such that $\f= \vf_{\la,a}$, for some $\la$ with $|\la|=1$, $a\in\D$. Then for all $b\in [0,|a|]$ there exists a point $c\in\D$ with  $|c|=b$, a constant $\mu$ with $|\mu|=1$, and an analytic function $G$ in the disk such that $W_{G, \vf_{\mu,c}}$ is also a co-isometric WCO on $\ch$.
\end{lem}
\begin{proof}
It is easy to check directly that the product of two co-isometric operators is again co-isometric. Thus, whenever $|\ta|=1$ and $W_{F,\f}$ is co-isometric, the operator $C_{R_\ta} W_{F,\la \vf_a}$ is also co-isometric. In view of the simple identity for compositions of automorphisms:
$$
 \vf_{\la,a} (\ta z) = \vf_{\ta \la,\overline{\ta} a}(z)\,,
$$
we obtain the following operator identity:
$$
 C_{R_\ta} W_{F,\vf_{\la,a}} = W_{C_{R_\ta F},\,\vf_{\ta \la,\,\overline{\ta} a}}\,,
$$
so the latter WCO is a co-isometric operator for every value of $\ta$ with $|\ta|=1$. Thus, its square
$$
 W_{C_{R_\ta F},\,\vf_{\ta \la, \overline{\ta} a}} W_{C_{R_\ta F},\,\vf_{\ta \la, \overline{\ta} a}} = W_{G,\,\vf_{\mu,c}}
$$
is also co-isometric, where
\begin{equation}
 G(z) = F(\ta z) \cdot F(\vf_{\ta^2\la,\overline{\ta} a}(z))  \,, \qquad \vf_{\mu,c} = \vf_{\ta \la,\,\overline{\ta} a} \circ \vf_{\ta \la,\,\overline{\ta} a}\,.
 \label{eqn-comp}
\end{equation}
(Note that the last map must equal some $\vf_{\mu,c}$ for some $c\in\D$ and some $\mu$ with $|\mu|=1$ since it is again a disk automorphism.)
\par
Next, let $b$ be an arbitrary number such that $0<b<|a|$. Let $F$ and $\f$ be fixed as above. We have the freedom to choose $\ta$ arbitrarily in \eqref{eqn-comp} and will now show that there exists $c$ as above (with $\ta$ chosen appropriately) so that $|c|=b$. In fact, such value of $c$ from the conditions above can easily be computed explicitly; indeed, we must have
$$
 (\vf_{\ta \la,\overline{\ta} a} \circ \vf_{\ta \la,\overline{\ta} a}) (c) = 0\,,
$$
hence $\vf_{\ta \la,\overline{\ta} a} (c) = \overline{\ta} a$, meaning that
$$
 \ta \la \frac{\overline{\ta} a - c}{1 - \overline{a} \ta c} = \overline{\ta} a\,,
$$
which yields
$$
 c = \overline{\ta} a \frac{\la \ta - 1}{\la \ta - |a|^2}\,.
$$
Now note that
$$
 |c| = |a| \frac{|\la \ta - 1|}{|\la \ta - |a|^2|}
$$
is a continuous function of the complex parameter $\ta$. When $\ta = \overline{\la}$, this function takes on value $0$ while choosing $\ta = - \overline{\la}$ yields
$$
  \frac{2 |a|}{1 + |a|^2} > |a|
$$
as the value of the function (recall that $a\neq 0$ since $\f$ is not a rotation). Since the complex parameter $\ta=e^{i t}$ can in turn be viewed as a continuous function of the real variable $t\in [0,2\pi]$, the elementary Bolzano's intermediate value theorem from Calculus implies that there exists a value $\ta$ with $|\ta|=1$ such that $|c|=b$, as claimed. (Of course, the last argument could have been made more explicit.)
\end{proof}
\par\medskip
We already know from T. Le's work \cite{L} that if $\f$ is a disk automorphism and $F=\mu (\f^\prime)^{\g/2}$, where $\g=\g(1)$ and $\mu$ is a constant such that $|\mu|=1$, then the induced operator  $W_{F,\f}$ is unitary (hence, also co-isometric) on the space $\ch_\g$. The next key statement identifies such spaces as the only ones on which $W_{F,\f}$ can be co-isometric if the automorphism $\f$ is not a  rotation. It should be remarked that, using the following conformal property of the kernel:
$$
 K_{\f(w)}(\f(z)) = \frac{K_a(a) K_w(z)}{K_a(z) K_w(a)}\,,
$$
for all disk automorphisms $\f$ and $a=\f^{-1}(0)$, the theorem below can also be deduced from a characterization of the spaces $\ch_\g$ given by \cite[Proposition~4.3]{H}.
\par\medskip
\begin{thm} \label{thm-unbded-ker}
Let $\ch$ be a weighted Hardy space that satisfies axioms (\textbf{A1})--(\textbf{A3}) and whose reproducing kernel is unbounded on the diagonal, and let $W_{F,\f}$ be a co-isometric WCO on $\ch$.
If the automorphism $\,\f\/$ is not a rotation, then there exists a positive $\g$ (namely, $\g=\g(1)$ in the formula for the kernel) such that $\ch=\ch_\g$, a space whose coefficients $\g(n)$ satisfy \eqref{eqn-coeff-wds} and whose kernel, thus, is of the form \eqref{eqn-spec-ker}. Moreover, $F = \mu (\f^\prime)^{\g/2} = \nu \frac{K_a}{\|K_a\|}$, where $a=\f^{-1}(0)$ and $\mu$ and $\nu$ are  constants such that $|\mu|=|\nu|=1$.

\end{thm}
\begin{proof}
Let $\f = \vf_{\la,a}$, for some $a\in\D$, $a\neq 0$, and $\la$ with $|\la|=1$. Put $w=a$ in \eqref{eqn-F-fi} and recall that $\f(a)=0$ to obtain
$$
 F(z) = \frac{K_a(z)}{\overline{F(a)}}\,.
$$
In what follows, we will always write simply $K_a^\prime(z)$ instead of $\dfrac{\pd K_a}{\pd z}(z)$. After differentiation with respect to $z$, we obtain
$$
  F^\prime(z) = \frac{K_a^\prime(z)}{\overline{F(a)}}\,.
$$
Also, note that $F(0)=1/\overline{F(a)}$.
\par
On the other hand, recalling that $\f(0)=\la a$ and differentiating \eqref{eqn-F} with respect to $z$, we get
$$
  F^\prime(z) = \frac{-\f^\prime(z) K_{\la a}^\prime(\f(z))}{\overline{F(0)} K_{\la a}^2(\f(z))} =  \frac{\la (1- |a|^2) K_{\la a}^\prime(\f(z))}{\overline{F(0)} (1-\overline{a}z)^2  K_{\la a}^2(\f(z))} \,.
$$
Equating the right-hand sides of the last two equations, taking also into account the fact that $F(0)=1/\overline{F(a)}$, yields
$$
 \frac{|F(a)|^2 (1-|a|^2)}{(1-\overline{a}z)^2} K_{\la a}^\prime(\f(z)) = \overline{\la} K_{\la a}^2(\f(z)) K_a^\prime(z)\,.
$$
Setting $z=a$, we obtain
\begin{equation}
 \frac{|F(a)|^2}{1-|a|^2} K_{\la a}^\prime(0) = \overline{\la} K_a^\prime(a)\,.
 \label{eq-F-a}
\end{equation}
Differentiation of the formula for the kernel \eqref{eqn-ker} with respect to $z$ yields
$$
 K_w^\prime (z) = \g(1)\overline{w} + \sum_{n=2}^\infty n \g(n) \overline{w}^n z^{n-1}\,,
$$
hence
$$
 K_a^\prime (a) = \g(1)\overline{a} + \sum_{n=2}^\infty n \g(n) \overline{a} |a|^{2(n-1)}\,, \qquad K_{\la a}^\prime(0) = \g(1) \overline{\la} \overline{a}\,.
$$
Bearing in mind that
$$
 |F(a)|^2 = K_a(a) = 1 + \sum_{n=1}^\infty \g(n) |a|^{2n}
$$
and using \eqref{eq-F-a}, it follows that
\begin{eqnarray*}
 1 + \sum_{n=1}^\infty \g(n) |a|^{2n} &=& (1-|a|^2) \( 1 + \sum_{n=2}^\infty \frac{n \g(n)}{\g(1)} |a|^{2(n-1)} \)
\\
 &=& 1 + \sum_{n=1}^\infty \frac{(n+1) \g(n+1)}{\g(1)} |a|^{2n} - \sum_{n=1}^\infty \frac{n \g(n)}{\g(1)} |a|^{2n}
\end{eqnarray*}
(after regrouping the terms). From here we obtain that
\begin{equation}
 \sum_{n=1}^\infty \g(n) (n+\g(1)) |a|^{2n} = \sum_{n=1}^\infty (n+1) \g(n+1) |a|^{2n}\,.
 \label{eq-a-c}
\end{equation}
Note that this holds only for one point $a$, for one given co-isometric  operator $W_{F,\,\f_{\la,a}}$. However, an application of Lemma~\ref{lem-autom-change} allows us to produce other WCOs $W_{G,\, \vf_{\mu,c}}$ that are also co-isometric and with different values $c$ instead of $a$ so as to include all possible values of $|c|$ with $0\le |c|\le |a|$, with conclusions analogous to \eqref{eq-a-c}. Note also an important point that, whenever $c\neq 0$, the symbol $\vf_{\mu,c}$ is not a rotation. By this construction, we obtain
$$
 \sum_{n=1}^\infty \g(n) (n+\g(1)) x^{2n} = \sum_{n=1}^\infty (n+1) \g(n+1) x^{2n}
$$
for all $x\in [0,|a|]$, the case $x=0$ being an obvious equality. The power series in the identity above are both even functions of $x$ that converge and coincide in $[-|a|,|a|]$. The uniqueness of the coefficients of a real power series in such an interval implies that
$$
 (n+1) \g(n+1) = \g(n) (n+\g(1))\,, \qquad \mathrm{ for \ all \ } n\ge 1\,.
$$
This recurrence equation is easily solved: since
$$
 \g(n+1) = \g(n) \frac{n+\g(1)}{n+1}\,, \qquad \mathrm{ for \ all \ } n\ge 1\,,
$$
and the formula trivially also extends to the case $n=0$ (by the fact that $\g(0)=1$), we have
$$
 \g(n) = \frac{n-1+\g(1)}{n} \g(n-1)\,, \qquad \mathrm{ for \ all \ } n\ge 1\,.
$$
Using the standard property $\G(x+1)=x\G(x)$ for all $x>0$, from here we obtain by induction the desired formula \eqref{eqn-coeff-wds}, with $\g=\g(1)$.
\par\smallskip
Next, we derive the formula for $F$ in terms of $\f^\prime$. From \eqref{eqn-F}, for the space is $\ch = \ch_\g$ and $\f=\vf_{\la,a}$, we have
$$
 F(z) =  \frac{1}{\overline{F(0)} K_{\f(0)}(\f(z))} =  \frac{1}{\overline{F(0)}} \(1-\overline{\f(0)} \f(z)\)^\g =
 \frac{1}{\overline{F(0)}} \frac{(1-|a|^2)^\g}{(1-\overline{a}z)^\g}  \,.
$$
For $z=0$ this shows that $|F(0)|= (1-|a|^2)^{\g/2}$, hence for appropriately chosen $\mu$ and $\nu$ with $|\mu|=|\nu|=1$ we obtain
$$
 F(z) =  \nu \frac{(1-|a|^2)^{\g/2}}{(1-\overline{a}z)^\g} = \nu \frac{K_a(z)}{\|K_a\|} = \mu (\f^\prime(z))^{\g/2}\,,
$$
which completes the proof.
\end{proof}
\par\medskip
Finally, putting together Theorem~\ref{thm-bded-ker}, Theorem~\ref{thm-phi-aut}, and Theorem~\ref{thm-unbded-ker} and the comments on the results of Le preceding Theorem~\ref{thm-unbded-ker}, we obtain the complete conclusions of Theorem~\ref{thm-main}.
\par\medskip
\section{Final remarks}
 \label{sec-rmrks}
A natural problem for further research would be to understand the isometric WCOs on the spaces considered. This is obviously a more difficult question. For example, it does not seem clear how one could obtain a formula like \eqref{eqn-F} in this case. Working with the assumption typical of the isometries: $\,W_{F,\f}\sp{\ast} W_{F,\f} = I$, one does not seem to get much more than the formula
$$
  W_{F,\f}\sp{\ast} K_w = \overline{F(w)} K_{\f(w)}
$$
already proved here. But this does not seem to imply in any obvious way a general formula for $W_{F,\f}\sp{\ast} f$, $f\in\ch$, nor a formula like \eqref{eqn-F}.
\par
Still a harder question would be to describe the normal weighted composition operators on the general Hilbert spaces of analytic functions. Such operators have been studied but not fully described even on $H^2$ in \cite{BN}, except in the case when the fixed point of the composition symbol belongs to the disk. Interesting general results have been obtained for the spaces $\ch_\g$ in \cite{L} and also for general spaces in \cite{Z2}. To the best of our knowledge, even in one variable, a complete answer is far from being known for the general family of spaces considered in this paper.
\par\medskip
\textsc{Acknowledgments}. {\small The first author was supported by the Viera y Clavijo program (2019/0000096) of ULL, Spain. The third author recieved partial support from the Erwin Schr\"odinger Institute, Vienna, during the workshop ``Operator Related Function Theory'' held in April of 2019. All authors acknowledge support from  PID2019-106870GB-I00 from MICINN, Spain. 
\par
The authors would like to thank Michael Hartz, Jos\'e A. Pel\'aez, and Nina Zorboska for some useful comments and references and Iason Efraimidis for pointing out some misprints in an earlier version. We are particularly grateful to the referee for pointing out several imprecise and incomplete details in an earlier version of the manuscript, for a number of useful comments and also for suggesting some alternative ideas of proofs.}


\end{document}